\documentclass[final,twoside,11pt]{entics} 
\usepackage{enticsmacro}
\usepackage{graphicx}
\usepackage[all]{xy}
\usepackage{lineno,hyperref}
\usepackage{amssymb}
\usepackage{amsmath}
\usepackage{dcolumn}
\usepackage{endnotes}
\usepackage{tabularx}
\newcommand{\mfx}{\mathsf{OWF}(X)}

\newcommand{\ua}{\mathord{\uparrow}}
\newcommand{\da}{\mathord{\downarrow}}

\newcommand{\cl}{{\rm cl}}

\newcommand\twoheaduparrow{\mathord{\rotatebox[origin=c]{90}{$\twoheadrightarrow$}}}
\newcommand\twoheaddownarrow{\mathord{\rotatebox[origin=c]{90}{$\twoheadleftarrow$}}}

\newcommand{\dda}{\twoheaddownarrow}
\newcommand{\dua}{\twoheaduparrow}

\newcommand{\II}{\begin{enumerate}}
	\newcommand{\III}{\end{enumerate}}

\def\conf{ISDT 2022} 	
\volume{2}			
\def\volu{2}			
\def\papno{11}			
\def\mydoi{{\normalshape  \href{https://doi.org/10.46298/entics.10392}{10.46298/entics.10392}}}
\def\copynames{C. Shen, X. Xi, D. Zhao} 
\def\runtitle{Further studies on  open well-filtered spaces}

\def\lastname{Shen, Xi and Zhao}
\begin{document}
\begin{frontmatter}
  \title{Further Studies on  Open Well-filtered Spaces} 
  \author{Chong Shen\thanksref{a}\thanksref{ALL}\thanksref{myemail}}	
   \author{Xiaoyong Xi\thanksref{b}\thanksref{1coemail}}	
    \author{Dongsheng Zhao\thanksref{c}\thanksref{2coemail}}	
   \address[a]{School of Science\\ Beijing University of Posts and Telecommunications\\				
    Beijing, China}  							
  \thanks[ALL]{This work was supported by the National Natural Science Foundation of China (1210010153, 12071188) and Jiangsu Provincial Department of Education (21KJB110008).}   
   \thanks[myemail]{Email: \href{shenchong0520@163.com} {\texttt{\normalshape
       shenchong0520@163.com}}} 
  \address[b]{School of Mathematics and Statistics\\Yancheng Terchers University\\
    Yancheng, Jiangsu Province, China} 
  \thanks[1coemail]{Email:  \href{littlebrook@sina.com} {\texttt{\normalshape
        littlebrook@sina.com}}}
     \address[c]{Mathematics and Mathematics Education, National Institute of Education\\
     	Nanyang Technological University\\
    	1 Nanyang Walk, Singapore} 
    \thanks[2coemail]{Email:  \href{dongsheng.zhao@nie.edu.sg} {\texttt{\normalshape
    			dongsheng.zhao@nie.edu.sg}}}
\begin{abstract} 
 The open well-filtered spaces were introduced by Shen, Xi, Xu and Zhao to answer the problem whether every core-compact  well-filtered space is sober.
 In the current paper we explore further properties of open well-filtered spaces. One of the main results is that if a space  is open well-filtered,
 then so is its upper space (the set of all nonempty saturated compact subsets equipped with the upper Vietoris topology). Some other properties on open well-filtered spaces are also studied.
\end{abstract}
\begin{keyword}
well-filtered space, sober space, core-compact space, locally compact space, open well-filtered space
\end{keyword}
\end{frontmatter}
\section{Introduction}\label{intro}
The open well-filtered spaces were introduced in \cite{shen-xi-xu-zhao-2020} and used to give an alternative and more natural proof of the conjectured conclusion that every core compact and well-filtered space is sober. In \cite{shen-xi-xu-zhao-2020}, the authors  actually proved a stronger conclusion: every core-compact and open well-filtered space is sober, as every well-filtered space is open well-filtered and the converse conclusion is not true in general. By \cite{xi-zhao-2017} and \cite{xu-xi-zhao-2021}, one  knows that  a space is well-filtered if and only if its upper space is well-filtered. It is thus natural to wonder  whether a similar result holds for open well-filteredness. In this paper we shall give a partial answer to this problem and prove that if a space is open well-filtered then so is its upper space. A number of  new results on open well-filtered spaces will also be presented here.

\section{Preliminaries}

We first recall  some basic definitions and results that will be used in the paper.

Let $P$ be a poset.
A nonempty subset $D$ of $P$ is \emph{directed} (resp., \emph{filtered}) if every two
elements of $D$ have an upper (resp., lower) bound in $D$.  A poset $P$ is  a
\emph{directed complete poset}, or a \emph{dcpo} for short, if for any
directed subset $D\subseteq P$, the supremum $\bigvee D$ exists.

For any subset $A$ of a poset $P$, 
$$\ua A=\{y\in P: \exists x\in A, x\leq y\}, {\mbox{ and } } \da A=\{y\in P:\exists x\in A, y\leq x\}.$$
In particular, for each $x\in X$, we write $\ua x=\ua\{x\}$ and  $\da x=\da \{x\}$.

For  $x, y\in P$,  $x$ is \emph{way-below} $y$, denoted by $x\ll  y$, if for any directed subset
$D$ of $P$ with $\bigvee D$ existing,
$y\leq\bigvee D$
implies $x\leq d$ for some $d\in D$.  Denote $\dua
x=\{y\in P:x\ll y\}$ and $\dda x=\{y\in P:y\ll x\}$. A poset $P$ is
\emph{continuous}, if for any $x\in P$, the set $\dda x$ is directed and
$x= \bigvee\dda x$. A continuous dcpo is also called a
\emph{domain}.

A subset $U$ of a poset $P$ is \emph{Scott open} if
(i) $U=\mathord{\uparrow}U$ and (ii) for any directed subset $D$ of $P$ with  $\bigvee D$ existing, $\bigvee D\in U$ implies $D\cap
U\neq\emptyset$. All Scott open subsets of $P$ form a topology on $P$,
called the \emph{Scott topology} on $P$ and
denote by $\sigma(P)$. The space $\Sigma P=(P,\sigma(P))$ is called the
\emph{Scott space} of $P$.

For a $T_0$ space $X$,
the specialization order $\leq$ on $X$ is defined by $x\leq y$ iff $x\in \cl_X(\{y\})$, where $\cl_X$ is
the  closure operator of $X$. Clearly, $\cl_X(\{y\})=\da y$.

In  what follows, if no otherwise specified, the partial order on a $T_0$ space will mean the specialization order.

\begin{remark}\cite{redbook,goubault}\label{re}
	For each poset $(P,\leq_P)$, the specialization order on $\Sigma P$  coincides with  $\leq_P$.
\end{remark}

For any  $T_0$ space $X$, we shall often use $\mathcal O(X)$ to denote the topology of $X$ (the collection of all open sets of $X$).  For any subset $A$ of  $X$, the  \emph{saturation} of $A$, denoted by $Sat_X(A)$, is defined by
$$ Sat_X(A)=\bigcap\{U\in \mathcal O(X): A\subseteq U\},$$
or equivalently, $Sat_X(A)=\ua A$ with respect to the specialization order (see \cite[Proposition 4.2.9]{goubault}).
A subset $A$ of  $X$ is \emph{saturated} if $A=Sat_X(A)$.

For any $U,V\in\mathcal O(X)$, we write $U\ll V$ for that $U$ is way-below $V$ in the poset $(\mathcal O(X),\subseteq)$. Using a similar proof to that of the Alexander's Subbase Lemma (see \cite[Theorem 4.4.29]{goubault}), we can obtain the following result.

\begin{lemma}\label{lem01}
	Let $X$ be a $T_0$ space, $\mathcal S$ be a subbase for $\mathcal O(X)$, and $U,V\in\mathcal O(X)$. Then $U\ll V$ if and only if one can extract a finite subcover of $U$ from any cover $\{U_i:i\in I\}\subseteq\mathcal S$ of $V$.
\end{lemma}

\begin{definition}\cite{shen-xi-xu-zhao-2020}
	Let $X$ be a $T_0$ space.
	\begin{enumerate}	
		\item[(1)] A subfamily $\mathcal F\subseteq \mathcal O(X)$ is called \emph{$\ll$-filtered}, denoted by $\mathcal F\subseteq_{flt}\mathcal O(X)$, if for any  $U_1,U_2\in\mathcal F$, there exists $U_3\in\mathcal F$ such that
		$U_3\ll U_1$ and $U_3\ll U_2$.
		
		\item [(2)] $X$ is called \emph{open well-filtered} if for each $\ll$-filtered family $\mathcal F\subseteq \mathcal O(X)$  and $U\in\mathcal O(X)$, $\bigcap\mathcal F\subseteq U$ implies that $V\subseteq U$ for some $V\in\mathcal F.$
	\end{enumerate}
\end{definition}

\begin{proposition}\cite{shen-xi-xu-zhao-2020}\label{prop2}
	If $X$ is an open well-filtered space and $\{U_i:i\in I\}$ is a $\ll$-filtered family of nonempty open sets, then $\bigcap\{U_i: i\in I\}$ is a nonempty compact saturated set.
\end{proposition}

\begin{definition}\cite{redbook,goubault}
	A $T_0$ space $X$ is called \emph{core-compact} if for each $x\in X$ and $U\in \mathcal O(X)$ such that $x\in U$, there exists $V\in\mathcal O(X)$ such that  $x\in V\ll U$.
\end{definition}

\begin{remark}\cite{redbook,goubault}
	A $T_0$ space $X$ is core-compact if and only if the poset  $(\mathcal O(X),\subseteq)$ is continuous.
\end{remark}

\begin{theorem}\cite{shen-xi-xu-zhao-2020}
	Every core-compact open well-filtered space is sober.
\end{theorem}

\begin{definition}\cite{redbook,goubault} A $T_0$ space $X$ is called \emph{well-filtered} if for any filtered family $\{Q_i:i\in I\}$ of compact saturated subsets of $X$ and any open set $U\subseteq X$, $\bigcap_{i\in I}Q_i\subseteq U$ implies $Q_{i_0}\subseteq U$ for some $i_0\in I$.
\end{definition}

\begin{definition}
	\cite{redbook,goubault} A nonempty subset $A$ of a topological space $X$ is called \emph{irreducible} if for any closed sets $F_1$, $F_2$ of $X$, $A \subseteq F_1\cup F_2$ implies $A \subseteq F_1$ or $A \subseteq  F_2$.  A $T_0$ space $X$ is called \emph{sober}, if for any  irreducible closed set $F$ of $X$, there is a (unique) point $x\in X$ such that $F=\cl_X(\{x\})$.
\end{definition}

The following result on irreducible sets in product spaces will be used in the sequel.
\begin{proposition}
	\cite[Proposition 8.4.7]{goubault} \label{closure}
	Let	$\{X_i:i\in I\}$ be a family of topological spaces. The irreducible closed sets in $\prod_{i\in I}X_i$ are exactly the sets of the form
	$\prod_{i\in I}C_i$, where each $C_i$ is an irreducible closed set in $X_i$ ($i\in I$).
\end{proposition}

The relations among well-filtered spaces, open well-filtered spaces and sober spaces are shown below:
$$\xymatrix{
	&\mbox{sober}\ar@{->}[ld]&\\
	\mbox{well-filtered} \ar@{->}[rr]&&\mbox{open well-filtered }\ar@{->}[ul]_{\mbox{core-compact}}
}$$

\vspace{0.2cm}

The following lemma will be used in the sequel.
\begin{lemma} 
	\cite[Lemma 2.4]{shen-xi-xu-zhao-2020} \label{llrudin}
	Let $X$ be a  $T_0$ space, $\{U_i:i\in I\}$ be a $\ll$-filtered family of open sets of $X$, and $F$ be a closed set of $X$. If  $F\cap U_i\not=\emptyset$ for all $i\in I$, then there is a minimal closed set $F_0\subseteq F$ such that $F_0\cap U_i\not=\emptyset$ for all $i\in I$. In addition, this $F_0$ is irreducible.
\end{lemma}

\section{A new characterization for open well-filtered spaces}

In \cite{shenxixuzhao2019,Xu-Shen-Xi-Zhao},  a characterization for  well-filtered spaces by means of KF-sets is obtained.  We now prove a similar characterization for open well-filtered spaces.

\begin{definition}
	A nonempty subset  $A$  of a $T_0$ space $X$ is called an \emph{open well-filtered set}, or \emph{OWF-set} for short, if there exists  a $\ll$-filtered family $\{U_i:i\in I\}\subseteq \mathcal O(X)$ such that $\cl(A)$ is a minimal closed set that intersects all $U_i$, $i\in I$.
	
	Denote by $\mathsf{OWF}(X)$ the set of all closed OWF-subsets of $X$.	
\end{definition}

\begin{remark}\label{re32}
	\begin{enumerate}
		\item [(1)] A subset of a topological space is an OWF-set if and only if  its closure is an OWF-set.
		\item [(2)] By Lemma \ref{llrudin}, every OWF-set is irreducible.
	\end{enumerate}	
\end{remark}

\begin{theorem}\label{th1}
	Let $X$ be a  $T_0$ space. Then the following statements are equivalent:
	\begin{enumerate}
		\item [\rm (1)] $X$ is open well-filtered;
		\item [\rm (2)] $\forall A\in \mfx$, there exists a unique $x\in X$ such that $A=\cl(\{x\})$.
	\end{enumerate}
\end{theorem}
\begin{proof}
	(1) $\Rightarrow$ (2). Let $A\in \mfx$. Then there exists $\{U_i:i\in I\}\subseteq_{flt}\mathcal O(X)$ such that $A$ is a minimal closed set that intersects all $U_i$, $i\in I$. Since  $X$ is open well-filtered, it follows that $\bigcap_{i\in I}U_i\cap A\neq\emptyset$.
	Choose one $x\in \bigcap_{i\in I}U_i\cap A$. Then $\cl(\{x\})\subseteq A$, and it is a closed set that intersects all  $U_i$, $i\in I$. By  the minimality of $A$, we have that  $A=\cl(\{x\})$. The uniqueness of $x$ is determined by the $T_0$ separation of $X$.
	
	(2) $\Rightarrow$ (1). Let $\{U_i:i\in I\}\subseteq_{flt}\mathcal O(X)$ and $U\in\mathcal O(X)$ such that $\bigcap_{i\in I}U_i\subseteq U$.
	We need to show that $U_i\subseteq U$ for some $i\in I$. If, on the contrary, that $U_i\nsubseteq U$ for all $i\in I$, then 
	by Lemma \ref{llrudin}, there exists a minimal (irreducible) closed set $A\subseteq X\setminus U$ that intersects all $U_i$, $i\in I$. Thus $A\in \mfx$. By the  assumption, there exists a unique $x\in X$ such that $A=\cl(\{x\})$. For each $i\in I$, since $\cl(\{x\})\cap U_i=A\cap U_i\neq\emptyset$, it follows that $x\in U_i$, which implies that  $x\in\bigcap_{i\in I}U_i\subseteq U$. Thus $x\in U\cap A$, which contradicts $A\subseteq X\setminus U$. This contradiction completes the proof. 
\end{proof}

\begin{example}\label{e01}
	Let $\mathbb N^+$ be the set of all positive integers, $\mathbb N^+_{\rm cof}$ be the space of $\mathbb N^+$ with the co-finite topology (the open sets are $\emptyset$  and all the complements of finite subsets of $\mathbb N^+$), and let $\mathbb N^+_{\alpha}$ be the Alexandoff space of $\mathbb N^+$ (the open sets are the upper subsets of $\mathbb N^+$ with the usual order of numbers). We have the following claims:
	\begin{enumerate}
		\item [(c1)] It is trivial to check that
		each subset of $\mathbb N^+$ is compact  in both $\mathbb N^+_{\rm cof}$ and $\mathbb N^+_{\alpha}$. We then deduce that $U\ll V$ iff $U\subseteq V$ for any open sets $U,V$ in $\mathbb N^+_{\rm cof}$ or $\mathbb N^+_{\alpha}$.
		
		\item [(c2)] Neither $\mathbb N^+_{cof}$ nor $\mathbb N^+_{\alpha}$ is open well-filtered.
		
		For each $n\in\mathbb N^+$,  $U_n=\mathbb N^+\setminus\{1,2,3,\ldots,n\}$ is a nonempty open set in both $\mathbb N^+_{\rm cof}$ and $\mathbb N^+_{\alpha}$. From (c1), it follows that the family  $\{U_n:n\in\mathbb N^+\}$ is  $\ll$-filtered, but $\bigcap_{n\in\mathbb N^+}U_n=\emptyset$. By Proposition \ref{prop2}, we obtain (c2).
	\end{enumerate}
\end{example}

The following example shows that a saturated subspace of an open well-filtered space need not be open well-filtered.

\begin{example}\label{exm3}
	Let  $\mathbb J=\mathbb N^+\times (\mathbb N^+\cup\{\omega\})$ be the Johnstone's dcpo \cite{goubault,john}, which is  ordered by
	$(m,n)\leq (m',n')$ iff either $m=m'$ and $n\leq n'$, or $n'=\omega$ and $n\leq m'$ (refer to Figure 1).
	\begin{figure}[htbp]
		\centering
		\includegraphics[scale=.6]{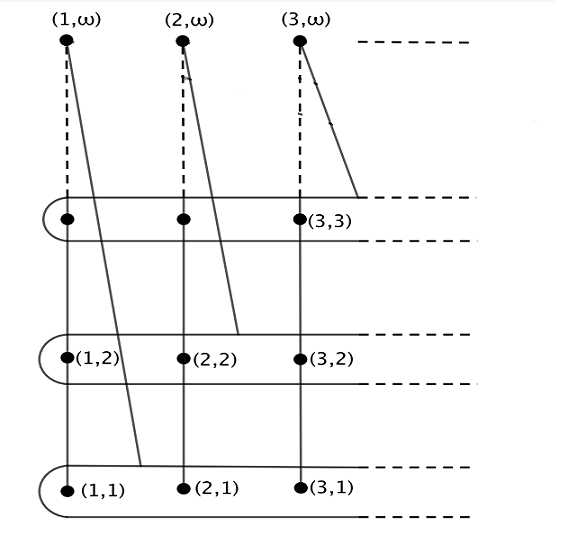}
		\caption{The Johnstone's dcpo $\mathbb J$}
	\end{figure}
	
	We have the following conclusions.
	\begin{enumerate}\label{exm2}
		\item [(r1)] $\Sigma \mathbb J$ is open well-filtered.
		
		Note that $ \forall U,V\in \sigma(\mathbb J)$, $U\ll V$ iff $U=\emptyset$ (see \cite[Exercise 5.2.15]{goubault}), which implies that each $\ll$-filtered family $\mathcal F$ of $\sigma(\mathbb J)$ is equal to $\{\emptyset,U\}$, where $U$ is an arbitrary Scott open set in $\mathbb J$. This means there exists no OWF-set in $\Sigma \mathbb J$. By Theorem \ref{th1}, we deduce that $\Sigma \mathbb J$ is open well-filtered.
		
		\item [(r2)] The set of maximal points $\mathbb N^+\times\{\omega\}$, as a saturated subspace of $\Sigma\mathbb J$, is
		homeomorphic to $\mathbb N^+_{cof}$ of Example \ref{e01}, and thus  is not open well-filtered.
	\end{enumerate}			
	
	From (r1), we have that $\da x$ (which is exactly the closure of $\{x\}$ in $\Sigma\mathbb J$) is not an OWF-set for each $x\in J$. We then deduce that the closure of singletons need not be OWF-sets. From (r2), it follows that
	the saturated subspace of an open well-filtered spaces need not be open well-filtered.
\end{example}

\begin{remark}\label{rem11}
	From Example \ref{exm3}, we deduce that if each $\ll$-filtered family of open sets in a space $X$ contains the empty set, then $X$ must be open well-filtered.
\end{remark}

The following example shows that neither the closed subspace nor the  retract of an open well-filtered space is open well-filtered in general.
\begin{example}
	\cite[Example 4.13]{shen-xi-xu-zhao-2020}
	\label{exam}
	Let $P=\mathbb J\cup \mathbb N^+$, where $\mathbb J$ is the Johnstone's dcpo. For any $x,y\in P$, define $x\leq y$ if one of the following conditions holds (refer to Figure 2):
	\begin{enumerate}
		\item [(i)] $x,y\in \mathbb N^+$ and $x\leq y$ in $\mathbb N^+$ with the usual ordering;
		\item [(ii)] $x,y\in \mathbb J$ and $x\leq y$ in $\mathbb J$;
		\item [(iii)] $x\in \mathbb N^+$ and $y= (x,\omega)$.
	\end{enumerate}
	\begin{figure}
		\centering
		\includegraphics[scale=0.55]{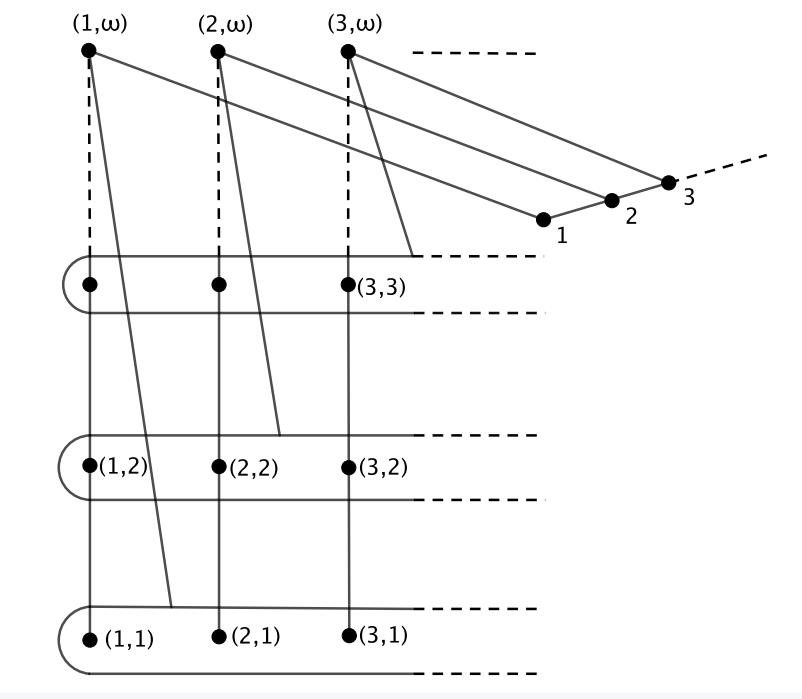}
		\caption{The poset $P$ of Example \ref{exam}}
	\end{figure}
	We have the following conclusions.
	\begin{enumerate}
		\item [(1)] $\Sigma P$ is open well-filtered (see \cite[Example 4.13]{shen-xi-xu-zhao-2020} for details).
		\item [(2)]  $\mathbb N^+$ is a Scott closed subset of $P$, and as a subspace of $P$, is homeomorphic to $\mathbb N^+_{\alpha}$, hence is not open well-filtered (by Example \ref{e01} (2)).
		
		\item [(3)] $\mathbb N^+$, as a subspace of $P$, is a retract of $P$.

		Let $e:\mathbb N^+\longrightarrow P$ be the identity embedding, and define $r:P \longrightarrow\mathbb N^+$ as follows: $\forall x\in P$, $\forall n\in\mathbb N^+$, $r(x)=n$ iff
		$$\left\{\begin{array}{lll}
			x\in\da(1,\omega),& \mbox{ when }n=1;\\
			x\in\da (n,\omega)\setminus\da (n-1,\omega),& \mbox{ when } n\geq 2.
		\end{array}\right.
		$$
		Since $\mathbb N^+$ is  a closed set in $\Sigma P$, we have that  $e$ is continuous, and note that $r^{-1}(\da n)=\da (n,\omega)$ is Scott closed for each $n\in\mathbb N^+$, which implies that $r$ is a continuous mapping. It is clear that the composition $r\circ e$ is the identity mapping on $\mathbb N^+$. Thus (3) holds.
	\end{enumerate}
	
	From above (1)--(3), we deduce that the closed subspace or the retract of an open well-filtered space need not be open well-filtered.	
\end{example}

\begin{proposition}\label{th}
	Let $X$ be a core-compact space. Then every irreducible set in $X$ is an OWF-set.
\end{proposition}
\begin{proof}
	Suppose that $A$ is an irreducible subset of $X$. Let
	$$\mathcal F=\{U\in\mathcal O(X): A\cap U\neq\emptyset\}.$$
	
\medskip
	
\emph{Claim 1:}  $\mathcal F\subseteq_{flt}\mathcal O(X)$.
	
	Let $U_1, U_2\in\mathcal F$. Then $A\cap U_1\neq\emptyset$ and $A\cap U_2\neq \emptyset$, implying that $A\cap U_1\cap U_2\neq\emptyset$.
	Take $x\in A\cap U_1\cap U_2$. Since $X$ is core-compact, there exists $U_3\in\mathcal O(X)$ such that $x\in U_3\ll U_1\cap U_2$. Note that $x\in U_3\cap A\neq\emptyset$, implying that $U_3\in\mathcal F$. Thus $\mathcal F$ is $\ll$-filtered.
	
	\medskip
	
\emph{Claim 2: } $\cl(A)$ is a minimal closed set that intersects all members of $\mathcal F$.
	
	Suppose $B$ is a closed set such that $B\subseteq \cl(A)$ and $B\cap U\neq\emptyset$ for all $U\in\mathcal F$.  We need to prove $\cl(A)\subseteq B$. Otherwise, $\cl(A)\nsubseteq B$, which implies  $A\cap (X\setminus B)\neq\emptyset$. Thus $X\setminus B\in\mathcal F$, which contradicts that
	$B$ intersects all members of $\mathcal F$. Therefore, $\cl(A)=B$.
	
	All this shows that $\cl(A)\in\mfx$.
\end{proof}

The following result is immediate by using Proposition \ref{th}.

\begin{corollary}
	\cite{LWX2020,Xu-Shen-Xi-Zhao}
	Every core-compact open well-filtered space is sober.
\end{corollary}

The following example shows that the continuous image of an OWF-set need not be an OWF-set.
\begin{example}
	Let $X$ be the Alexandroff space of the Johnstone's dcpo $\mathbb J$ (whose open sets are  the upper sets), and let $f:X\longrightarrow\Sigma \mathbb J$ be the identity mapping. Clearly, $f$ is a continuous mapping. Note that for any $x\in \mathbb J$,  since $X$ is locally compact (hence core-compact), by Proposition \ref{th},
	$\da x$ is an OWF-set in $X$, but it is not an OWF-subset of $\Sigma \mathbb J$ by Example \ref{exm3} (r1).
\end{example}

For a $T_0$ space $X$ and  $Y\in\mathcal O(X)$, if
$U,V\in\mathcal O(Y)$, then we have that $U,V\in\mathcal O(X)$, and that $U\ll V$ in $(\mathcal O(Y),\subseteq)$ if and only if $U\ll V$ in $(\mathcal O(X),\subseteq)$. Using this fact, one can prove the following result easily.

\begin{proposition}
	Every open subspace of an open well-filtered space is also open well-filtered.
\end{proposition}

\begin{definition}
	\cite{goubault}
	A $T_0$ space $X$ is called \emph{core-coherent} if for any  $U, V, W\in\mathcal O(X)$, $U\ll V$ implies that $U\cap W\ll V\cap W$.
\end{definition}

Example \ref{exm2} shows that singletons need not be OWF-sets. As a remedy, we have the following result.
\begin{theorem}
	Let $X$ be a core-coherent space. Then the following statements are equivalent:
	\begin{enumerate} 
		\item [\rm (1)] $X$ is core-compact;
		\item  [\rm (2)]all irreducible sets in $X$ are OWF-sets;
		\item [\rm (3)] the  singletons are OWF-sets.
	\end{enumerate}
\end{theorem}
\begin{proof}
	By Proposition \ref{th}, that (1) $\Rightarrow$ (2) is clear, and since singletons are irreducible, we obtain that (2) $\Rightarrow$ (3).
	Now we prove (3) $\Rightarrow$ (1).
	
	Let $x\in X$ and $U$ be an open neighborhood of $x$. Since $\{x\}$ is an OWF-set, there exists a $\ll$-filtered family $\{U_i:i\in I\}\subseteq\mathcal O(X)$ such that  $\cl(\{x\})$ is a minimal closed set that intersects all $U_i$, $i\in I$. Fix an $i_0\in I$. It follows that $x\in U_{i_0}$.
	Then there exists ${i_1}\in I$ such that $x\in U_{i_1}\ll U_{i_0}$. Since $X$ is core-coherent, it holds that $x\in U\cap U_{i_1}\ll U\cap U_{i_0}\subseteq U$. Therefore,
	$X$ is core-compact.
\end{proof}

It is easy to verify the following lemma.
\begin{lemma}\label{preserve}
	Let $f:X\longrightarrow Y$ be a continuous  open mapping between $T_0$ spaces, and $U,V\in\mathcal O(X)$. If $U\ll V$, then $f(U)\ll f(V)$.
\end{lemma}

Regarding the product spaces,  we are still not able to prove that the product of two open well-filtered spaces  is a open well-filtered space. Here is a result for some special spaces.
\begin{theorem}
	For each $T_0$ space $X$, the product $X\times \Sigma\mathbb J$ is open well-filtered.
\end{theorem}
\begin{proof}
	Let $U,V\in \mathcal O(X\times \Sigma\mathbb J)$ such that $U\ll V$.
	We prove that $U=\emptyset$.
	Note that the projection $p_2$ is a continuous open mapping, so $p_2(U), p_2(V) \in \sigma(\mathbb J)$ and by Lemma \ref{preserve}, $p_2(U)\ll p_2(V)$. Thus $p_2(U)=\emptyset$, which implies  that  $U=\emptyset$. By Remark \ref{rem11}, $X\times \Sigma\mathbb J$ is an open well-filtered space.
\end{proof}

The above theorem indicates that the open well-filteredness of the product of spaces does not imply the open well-filteredness of the  factor spaces. 
In the following, we will show that the open well-filteredness of the product of finite $T_0$ spaces implies that one of the factor spaces is open well-filtered.

It also is trivial to verify the following lemma.
\begin{lemma}\label{lem316}
	Let $\{X_k:1\leq k\leq n\}$ be a finite family of $T_0$ spaces, and $U_k,V_k\in \mathcal O(X)$ such that $U_k\ll V_k$ for $1\leq k\leq n$. Then $\prod_{1\leq k\leq n}U_k\ll\prod_{1\leq k\leq n}V_k$.
\end{lemma}

\begin{theorem}
	Let $\{X_k:1\leq k\leq n\}$ be a finite family of $T_0$ spaces, and $X$ be their product space. If $X$ is an open well-filtered space, then there is $k_0$ ($1\leq k_0\leq n$) such that $X_{k_0}$ is open well-filtered.
\end{theorem}
\begin{proof}
	Suppose on the contrary that every $X_k$ is not open well-filtered, $1\leq k\leq n$. Then there exist a $\ll$-filtered family $\mathcal F_k\subseteq\mathcal O(X_k)$ and an open set $O_k$ in $X_k$ such that $\bigcap\mathcal F_k\subseteq O_k$, but $U\nsubseteq O_k$ for any $U\in\mathcal F_k$. Define
	\begin{center}
		$\mathcal F=\left\{\prod_{1\leq k\leq n}U_k: \forall k, U_k\in\mathcal F_k\right\}.$
	\end{center}
	By Lemma \ref{lem316}, one can deduce that $\mathcal F$ is a $\ll$-filtered family of $\mathcal O(X)$ such that $\bigcap\mathcal F\subseteq \prod_{1\leq k\leq n}O_k\in\mathcal O(X)$.
	Since $X$ is open well-filtered, there exists $\prod_{1\leq k\leq n}U_k\in\mathcal F$ (i.e., $U_k\in\mathcal F_k$ for $1\leq k\leq n$) such that $\prod_{1\leq k\leq n}U_k\subseteq \prod_{1\leq k\leq n}O_k$. Note that each $U_k$  ($1\leq k\leq n$) is not empty,  which follows that $U_k\subseteq O_k$ for $1\leq k\leq n$, a contradiction.
\end{proof}

\section{Upper  spaces and open well-filteredness}

In this section, we prove that if a space $X$ is open well-filtered then its upper space (or the Smyth power space)  is also open well-filtered. The proof here makes use of a technique employed in \cite{lyu-chen-jia}.

For any topological space $X$,  we  use $\mathcal{D}(X)$ to
denote the set of all nonempty compact saturated subsets of $X$. The
\emph{upper Vietoris topology} on $\mathcal{D}(X)$ is  the topology
that has $\{\Box U: U\in\mathcal{O}(X)\}$ as a base,
where $ \Box U=\{K\in \mathcal{D}(X): K\subseteq U\}$. The set
$\mathcal{D}(X)$ equipped with the upper Vietoris topology is called  the \emph{upper space}
or  \emph{Smyth power space} of $X$. Note that
$\{\diamondsuit\,F: X\setminus F\in\mathcal{O}(X)\}$ is a base of the co-upper Vietoris topology, where
$\diamondsuit\,F=\{K\in \mathcal{D}(X): K\cap F\not=\emptyset\}$.

\begin{remark}\label{rm1}
	Let $X$ be a $T_0$ space, and $U, U_1, U_2\in\mathcal{O}(X)$.
	\begin{enumerate}
		\item[(1)] $\Box U_1\subseteq \Box U_2$ if and only if $U_1\subseteq U_2$.
		\item [(2)] For any $x\notin U$,
		$\diamondsuit \cl(\{x\})\cap\Box U=\emptyset$.
		\item [(3)] If $\Box U\subseteq \Box U_1\cup \Box  U_2$, then $\Box U\subseteq \Box U_1$ or
		$\Box U\subseteq \Box U_2$ \cite{lyu-chen-jia}.
	\end{enumerate}
\end{remark}

We now state and prove the main result in this section.

\begin{theorem}
	For any open well-filtered space $X$, the upper space $\mathcal{D}(X)$ is open well-filtered.
\end{theorem}

\begin{proof}Assume that  $X$ is an open well-filtered space.
	
	Let $\{\mathcal{U}_{i}:i\in I\}$ be a $\ll$-filtered family of open sets $\mathcal{U}_i$ in $\mathcal{D}(X)$ and
	$$\bigcap\{\mathcal{U}_{i}: i\in I\}\subseteq \mathcal{U}$$
	for an open set $\mathcal{U}$ in $\mathcal{D}(X)$.
	
	For each $i\in I$, let $\mathcal{U}_{i}=\bigcup\{\Box U_{i, t}: t\in T_i\}$, where $U_{i, t}\in \mathcal{O}(X)$.
	
	Assume that for each $i\in I$, $\mathcal{U}_i\not\subseteq \mathcal{U}$. Let $\mathcal{C}=\mathcal{D}(X)\setminus \mathcal{U}$.
	Then $\mathcal{C}$ is a closed set  in $\mathcal{D}(X)$ such that $\mathcal C\cap \mathcal U_i\neq\emptyset$ for all $i\in I$. By Lemma \ref{llrudin}, there is a minimal closed set $\mathcal{C}_{0}\subseteq \mathcal{C}$ which also has a  nonempty intersection  with every $\mathcal{U}_{i}$.
	
	For each $i\in I$, let
	$$\hat{\mathcal{U}}_i =\bigcup\{\Box U_{i, t}: \Box U_{i, t}\cap \mathcal{C}_0\not=\emptyset,t\in T_i\}.$$
	
	\medskip
	
	\emph{ Fact 1:}  If $\mathcal{U}_{i_1} \ll  \mathcal{U}_{i_2} \ll \mathcal{U}_{i_3}$, then $\hat{\mathcal{U}}_{i_1} \ll  \hat{\mathcal{U}}_{i_3}$. Therefore,  $\{\hat{\mathcal{U}}_i:i\in I\}$ is $\ll$-filtered.
	
	We just need to verify the first statement.
	To see this, let $\{\mathcal{V}_{l}:l\in L\}$ be a directed open cover of $\hat{\mathcal{U}}_{i_3}$.
	Then $\{\mathcal{V}_{l}\cup (\mathcal D(X)\setminus\mathcal{C}_0):l\in L\}$ is a directed open cover of  $\mathcal{U}_{i_3}$.
	By the assumption, there is a $l_0\in L$ such that $\mathcal{U}_{i_2}\subseteq\mathcal{V}_{l_0}\cup (\mathcal D(X)\setminus\mathcal{C}_0)$.
	
	By $\mathcal{U}_{i_1} \ll  \mathcal{U}_{i_2}$ and the structure of the upper Vietoris topology, there exist
	$$ \Box  W_{1}, \Box  W_{2}, \cdots, \Box  W_{n} \mbox{ contained in } \mathcal{V}_{l_0},$$
	$$ \Box  G_{1}, \Box  G_{2}, \cdots, \Box  G_{m} \mbox{ contained in }\mathcal D(X)\setminus \mathcal{C}_{0}$$
	such that
	$$\mathcal{U}_{i_1}\subseteq \bigcup\{\Box  W_{k}: 1\leq k\leq n\}\cup\bigcup\{\Box  G_{h}:1\leq h\leq m\}.$$
	
	Note that $\hat{\mathcal{U}}_{i_1}\subseteq\mathcal U_{i_1}$. For each $t\in T_{i_1}$ such that  $\Box U_{i_1, t}\cap\mathcal C_0\neq\emptyset$, by Remark \ref{rm1}, $\Box U_{i_1, t}\subseteq \Box W_k$ for some $1\le k\le n$, or $\Box U_{i_1, t}\subseteq \Box G_h$  for some $1\le h\le m$. But $\Box G_h\cap \mathcal{C}_0=\emptyset$, hence
	$\Box U_{i_1, t}\subseteq \Box W_k\subseteq \mathcal{V}_{l_0}$ for some $k$. It then follows  that
	$$\hat{\mathcal{U}}_{i_1}\subseteq \mathcal{V}_{l_0}.$$
	Therefore $\hat{\mathcal{U}}_{i_1} \ll  \hat{\mathcal{U}}_{i_3}$ holds.
	
	\medskip
	
	For each $i\in I$, let
	$$U_i=\bigcup\{U_{i, t}: \Box U_{i, t}\cap \mathcal{C}_0\not=\emptyset\}.$$
	
	\medskip
	
	\emph{Fact 2: }  If $\hat{\mathcal{U}}_{i_1} \ll \hat{\mathcal{U}}_{i_2}$, then $U_{i_1} \ll  U_{i_2}.$ Hence
	$\{U_i:i\in I\}$ is  a $\ll$-filtered family of open sets in $X$.
	
	As a matter of fact, it is easy to see that if $\{W_j:j\in J\}\subseteq \mathcal{O}(X)$ is a directed open cover of $U_{i_2}$, then
	$\{\Box W_j: j\in J\}$ is a directed open cover of $\hat{\mathcal{U}}_{i_2}$, hence there is $j_0\in J$ such that $\hat{\mathcal{U}}_{i_1}\subseteq \Box W_{j_0}$, thus $U_{i_1}\subseteq W_{j_0}$.
	
	\medskip
	
	Let $$K=\bigcap\{U_{i}: i\in I\}.$$
	By Proposition \ref{prop2}, $K$ is a nonempty saturated compact set, that is $K\in\mathcal{D}(X)$.
	
	\medskip
	
	\emph{ Fact 3: } $K\not\in \mathcal{U}$.
	
	Indeed, if $K\in\mathcal{U}$, then there is an open set $E$ of $X$ such that
	$K\in\Box E\subseteq \mathcal{U}$. Then $K=\bigcap\{U_{i}: i\in I\}\subseteq E$, and since $X$ is open well-filtered, there is $i_0$ such that $U_{i_0}\subseteq E$.
	
	Choose one $U_{i_0, t_0}$ such that $\Box U_{i_0, t_0}\cap \mathcal{C}_0\not=\emptyset$. Then
	$U_{i_0, t_0}\subseteq U_{i_0}\subseteq E$, and it follows that $$\emptyset\neq\Box U_{i_0, t_0}\cap \mathcal{C}_0\subseteq\Box E\cap\mathcal C\subseteq\mathcal U\cap\mathcal C=\emptyset,$$ a contradiction.
	
	\medskip
	
	\emph{Fact 4: } $K\in \bigcap\{\hat{\mathcal{U}}_i: i\in I\}$.
	
	As a matter of fact, if the statement is not true, then there is $i_0\in I$ such that $K\not\in \hat{\mathcal{U}}_{i_0}$.
	By the definition of $\hat{\mathcal{U}}_{i_0}$. Take any $U_{i_0, t_0}$ with $\Box U_{i_0, t_0}\cap \mathcal{C}_0\not=\emptyset$.  Then $K\not\subseteq U_{i_0, t_0}$.
	Take an $e\in K\setminus U_{i_0, t_0}$ and let $\mathcal{F}=\diamondsuit \cl(\{e\})$. Then by Remark \ref{rm1} $\mathcal{F}\cap \Box U_{i_0, t_0}=\emptyset$.

	We show that $\mathcal{C}_0\cap \mathcal{F}\cap \mathcal{U}_i\not=\emptyset$ for all $i\in I$.  For this, it is enough to show that for any $i\in I$, there is a $t_i\in T_{i}$ such that
	$$\Box U_{i, t_i}\cap\mathcal{F}\cap\mathcal{C}_0\not=\emptyset.$$
	If not, there exists $i_{1}\in I$ such that
	
	$$\Box U_{i_1, t}\subseteq (\mathcal{D}(X)\setminus \mathcal{F})\cup (\mathcal{D}(X)\setminus \mathcal{C}_0)$$
	for all $t\in T_{i_1}$.  Choose $i_2$ such that $\hat{\mathcal{U}}_{i_2} \ll \hat{\mathcal{U}}_{i_1}$.
	Then there are
	
	$$ \Box  V^1_{1}, \Box  V^1_{2}, \cdots, \Box  V^1_{m} \mbox{ contained in } \mathcal{D}(X)\setminus \mathcal{C}_0,$$
	and
	$$ \Box  V^2_{1}, \Box  V^2_{2}, \cdots, \Box  V^2_{n} \mbox{ contained in } \mathcal{D}(X)\setminus \mathcal{F}$$
	such that
	$$\hat{\mathcal{U}}_{i_2}\subseteq \bigcup\{ \Box  V^1_{k}: 1\le k\le m\}\cup \bigcup\{ \Box  V^2_{l}: 1\le l\le n\}.$$
	
	By the definition of $K$,	we have that $e\in K\subseteq U_{i_2}$, and then there is $t'\in T_{i_2}$ such that $e\in U_{i_2, t'}$ where $\Box U_{i_2, t'}\cap \mathcal{C}_0\not=\emptyset$. Thus $Sat_X(\{e\})\in\mathcal F\cap \Box U_{i_2, t'}\neq\emptyset$.
	
	Now $\Box U_{i_2, t'}\subseteq \hat{\mathcal{U}}_{i_2}\subseteq \bigcup\{ \Box  V^1_{k}: 1\le k\le m\}\cup \bigcup\{ \Box  V^2_{l}: 1\le l\le n\}$.
	By Remark \ref{re}, we have that
	$$\Box U_{i_2, t'}\subseteq \Box V^1_{k} \mbox{ for some  } 1\leq k\leq m, \mbox{ or }\Box U_{i_2, t'}\subseteq \Box V^2_{l}\mbox{ for some  }1\leq l\leq n.$$
	But for any $k$, since $\Box U_{i_2, t'}\cap\mathcal C_0\neq\emptyset$ and
	$\Box V^{1}_{k}\cap \mathcal C_0=\emptyset$, it follows that $\Box U_{i_2, t'}\nsubseteq \Box V^1_{k}$, and for any $l$, since $\Box U_{i_2, t'}\cap\mathcal F\neq\emptyset$ and $\Box V^2_l\cap \mathcal{F}=\emptyset$, it follows that $\Box U_{i_2, t'}\nsubseteq \Box V^2_{l}$, a contradiction.
	This shows that $\mathcal{C}_0\cap \mathcal{F}\cap \mathcal{U}_i\not=\emptyset$ for all $i\in I$.
	
	By the minimality of $\mathcal{C}_0$, we  have that $\mathcal{C}_0\subseteq \mathcal{F}$.
	Then $\Box U_{i_0, t_0}\cap \mathcal{C}_0\subseteq \Box U_{i_0, t_0}\cap\mathcal{F}$. But, as we pointed out earlier,
	$\Box U_{i_0, t_0}\cap\mathcal{F}=\emptyset$, thus $\Box U_{i_0, t_0}\cap \mathcal{C}_0=\emptyset$. This contradicts that
	$\Box U_{i_0, t_0}\cap \mathcal{C}_0\not=\emptyset$.
	
	All these together show that $K\in \bigcap\{\hat{\mathcal{U}}_i: i\in I\}$.
	
	\medskip
	
	Now, Fact 3 and Fact 4 contradict the assumption
	$$\bigcap\{\mathcal{U}_{i\in I}: i\in I\}\subseteq \mathcal{U}.$$
	
	The proof is thus completed.
\end{proof}

\section{Summary}
In this paper, we mainly considered the preservation of the open well-filteredness by some standardly  constructed spaces from an open well-filtered space.
The table below  summarizes the main results, where ``sp." denotes ``subspaces".

\vspace{0.1cm}
\begin{center}
	
	\begin{tabular}{|c|c|c|c|c|c|}
		\hline
		open sp.&closed sp. &saturated sp. &retract&upper space& product \\
		\hline
		$\checkmark$& $\times$&$\times$ &$\times$ & $\checkmark$&$?$\\
		\hline
	\end{tabular}
\end{center}
\medskip

Note that  in  many other cases, the ground spaces usually  inherit the corresponding properties of their upper spaces as they can be embedded  into the upper spaces under  the principle filter mappings. At the moment, the following problem is still open.

\begin{problem}
	Is it true that a space is open well-filtered if its upper space is open well-filtered?
\end{problem}	
%
%


\bibliographystyle{./entics}

\end{document}